\documentclass[10pt]{article}
\usepackage[unicode]{hyperref}
\usepackage[utf8]{inputenc}
\usepackage[T1]{fontenc}
\usepackage[margin=1in]{geometry}
\usepackage{amsmath,amssymb,amsthm}
\usepackage{amsfonts}
\usepackage{calc}
\usepackage{tikz}
\usetikzlibrary{math, calc, patterns}
\usepackage{enumitem}
\newtheorem{theorem}{Theorem}[section]
\newtheorem*{theorem*}{Theorem}
\newtheorem{corollary}[theorem]{Corollary}
\newtheorem{lemma}[theorem]{Lemma}
\newtheorem{definition}[theorem]{Definition}
\newtheorem{proposition}[theorem]{Proposition}
\newtheorem{remark}[theorem]{Remark}
\newcommand{\N}{\mathbb{N}}
\newcommand{\R}{\mathbb{R}}
\DeclareMathOperator{\pos}{pos}
\DeclareMathOperator{\Next}{next}
\DeclareMathOperator{\Span}{span}
\DeclareFontFamily{U}{matha}{\hyphenchar\font45}
\DeclareFontShape{U}{matha}{m}{n}{
      <5> <6> <7> <8> <9> <10> gen * matha
      <10.95> matha10 <12> <14.4> <17.28> <20.74> <24.88> matha12}{}
\DeclareSymbolFont{matha}{U}{matha}{m}{n}
\DeclareFontSubstitution{U}{matha}{m}{n}
\DeclareFontFamily{U}{mathx}{\hyphenchar\font45}
\DeclareFontShape{U}{mathx}{m}{n}{
      <5> <6> <7> <8> <9> <10>
      <10.95> <12> <14.4> <17.28> <20.74> <24.88>
      mathx10}{}
\DeclareSymbolFont{mathx}{U}{mathx}{m}{n}
\DeclareFontSubstitution{U}{mathx}{m}{n}

\DeclareMathDelimiter{\vvvert}{0}{matha}{"7E}{mathx}{"17}
\newcommand{\vertiii}[1]{\left\vvvert #1 \right\vvvert}

\def\clap#1{\hbox to 0pt{\hss#1\hss}}

\def\mathclap{\mathpalette\mathclapinternal}

\def\mathclapinternal#1#2{%
\clap{$\mathsurround=0pt#1{#2}$}}

\title{The invariant subspace problem for the space of smooth functions on the real line}

\begin{document}

\author{Michał Goliński
\thanks{Faculty of Mathematics and Comp. Sci.,
Adam Mickiewicz University,
ul. Uniwersytetu Poznańskiego 4,
61-614 Poznań, POLAND}
\thanks{The authors research was partially
supported by National Science Centre (Poland)
Grant UMO-2013$\slash$10$\slash$A$\slash$ST1$\slash$00091.}
\thanks{\href{mailto:golinski@amu.edu.pl}{golinski@amu.edu.pl}} , 
Adam Przestacki
\setcounter{footnote}{0}\footnotemark \setcounter{footnote}{1} \footnotemark
\setcounter{footnote}{3}
\thanks{\href{mailto:adamp@amu.edu.pl}{adamp@amu.edu.pl}}
\thanks{Corresponding author}
}


\maketitle

\begin{abstract}
We construct a continuous
linear operator acting on the space of smooth functions
on the real line without non-trivial invariant subspaces.
This is a first example of such an operator
acting on a Fréchet
space without a continuous norm. The construction is based
on the ideas due to C. Read who constructed a continuous operator
without non-trivial invariant subspaces on the Banach space $\ell_1$.
\end{abstract}

\textit{Cyclic vectors} $\cdot$
\textit{Space of smooth functions} $\cdot$
\textit{Sequence spaces}

\section{Introduction}
Let $X$ be a locally convex space.
The invariant subspace (subset) problem is the question if
every continuous linear operator $T\colon X\to X$
has a non-trivial invariant subspace (subset),
i.e., if there exists a closed subspace (subset)
$0\subsetneq H\subsetneq X$ such that $T(H)\subset H$.
The case of the separable Hilbert space has been studied
by multiple authors and is one of the most important
open problems in operator theory.

In the Banach space setting the first counterexamples to the invariant subspace problem were constructed by P. Enflo \cite{MR0473871,MR892591}
and C. Read \cite{MR749447,MR806634} in the 1980s. While Enflo constructed an operator
on an artificial Banach space, Read was able to build his counterexample on $\ell_1$.
An accessible exposition of the Read's construction
can be found in the last chapter of \cite{MR2533318}.
Later on Read \cite{MR959046,MR950973} improved his methods and was able to
build counterexamples on other Banach spaces.

A. Atzmon in \cite{MR701260} published a construction of an operator on
a nuclear Fréchet space (with a continuous norm) without
non-trivial invariant subspaces.

The first author was able to adapt the methods of Read and constructed
counterexamples to the invariant subspace problem for many classical
Fréchet spaces including the space of holomorphic functions on the unit disc
$H(\mathbb{D})$ and the space of rapidly decreasing sequences $s$ (see \cite{MR2863862}, see
also \cite{MR3077885} for an operator without non-trivial invariant subsets on $s$).
The construction required the existence of a continuous norm on the underlying space.

In \cite{MR3866905} Q. Menet
was able to show that there is a
big family of Fréchet spaces with a continuous norm that
support an operator without non-trivial invariant subspaces (even subsets).

When the Fréchet space $X$ does not possess a continuous norm, then
there are two possibilities:
\begin{itemize}
  \item There exists a fundamental increasing sequence of seminorms
        $(p_n)_{n\in\N}$ for $X$ such that $\ker p_{n+1}$ is of finite codimension in $\ker p_n$
        for all $n$ (e.g., the space of all sequences $\omega$).
        Then every operator on $X$ has a non-trivial invariant subspace
        (see \cite[Theorem 2.1]{MR3866905}).
  \item No such fundamental system exists.
        In this case it is not clear whether an operator without non-trivial
        invariant subspaces exists -- this case is left open in \cite{MR3866905}.
\end{itemize}

In this paper we construct an operator without non-trivial
invariant subspaces on the space of smooth functions on the real line $C^\infty(\R)$
with the usual topology of uniform convergence of functions and their derivatives on compact sets.
This space does not possess a continuous norm and is isomorphic to
the countable product of the space of rapidly decreasing sequences $s^\N$
which plays an important role in the theory of
nuclear Fréchet spaces because of the celebrated Kōmura-Kōmura theorem.

\section{Preliminaries}
Throughout we will denote by $\N$ the set of non-negative integers.
For us an operator will always be a continuous linear map.
For all unexplained notions from functional analysis we refer to \cite{MR1483073}.
Below we describe the spaces which will be used in our construction.

\subsection{The space \texorpdfstring{$s$}{s} of rapidly decreasing sequences}
The space of rapidly decreasing sequences is the space
\[
    s
    = \left\lbrace
      \left(x_j\right)_{j=0}^\infty:
        p_N\left(\left(x_j\right)_{j=0}^\infty\right) =
         \sum_{j=0}^\infty |x_j|(j+1)^N<\infty
         \quad
         \text{for every $N\in \N$}
    \right\rbrace
\]
with the topology generated by the family of seminorms $\{p_N: N\in \N\}$.
Due to technical reasons we will use a different, although equivalent, system of seminorms for $s$.
The proof of the proposition below is standard (see, e.g.,
\cite[Section 2]{MR3077885}).
\begin{proposition}
\label{prop:matrix}
There exists a matrix $\displaystyle \left(A_{N,j}\right)_{N,j=0}^\infty$ such that:
\begin{enumerate}[leftmargin=2\parindent]
  \item We have $A_{N,j}\geq 1$ for every $N,j \in \N$.
  \item The sequence $\left(A_{N,j}\right)_{N=0}^\infty$ is increasing and unbounded for every $j\in \N$.
  \item The sequence $\left(A_{N,j}\right)_{j=0}^\infty$ is increasing and unbounded for every $N\in \N$.
  \item We have $\frac{A_{N,j+1}}{A_{N,j}}\leq 2$ for every $N,j \in \N$.
  \item We have $\displaystyle\lim_{j\to\infty} \frac{A_{N,j}}{A_{N+1,j}}=0$ for every $N \in \N$.
  \item We have $\displaystyle\lim_{j\to\infty} \frac{2^j}{A_{N,j}}=\infty$ for every $N \in \N$.
  \item The family of seminorms $\{p_N:N\in\N\}$ is equivalent to the family of seminorms $\{|\cdot|_N:N\in\N\}$, where
        \begin{equation*}
          \left|\left(x_j\right)_{j=0}^\infty \right|_N=\sum_{j=0}^\infty|x_j|A_{N,j}.
        \end{equation*}
        Moreover, the unit balls of the seminorms $|\cdot|_N$ form a basis of neighborhoods of zero in $s$.
\end{enumerate}
\end{proposition}

\begin{remark}
Observe that
in order to prove that a sequence $(x_n)_{n=0}^\infty$ converges to $x$ in $s$ one only needs
to prove that for every $N$ the inequality $|x_n-x|_N\leq 1$ holds for every $n$ big enough.
\end{remark}
\subsection{The space \texorpdfstring{$s^\N$}{sᴺ}}
The countable product of the space of rapidly decreasing sequences is the space
\[
    s^\N
    = \left\lbrace
      \left(x_{i,j}\right)_{i,j=0}^\infty:
        \left(x_{i,j}\right)_{i=0}^\infty \in s
        \quad
        \text{for every $j\in\mathbb{N}$}
    \right\rbrace
\]
with the natural Fréchet space topology of the product space.
This topology is generated by the family of seminorms $\{\|\cdot\|_N:N\in\N\}$, where
\[
  \left\|\left(x_{i,j}\right)_{i,j=0}^\infty\right\|_N
  = \sum_{j=0}^N \left|\left(x_{i,j}\right)_{\vphantom{j}i=0}^\infty \right|_N.
\]
Note that the unit balls of the seminorms $\|\cdot\|_N$ form
a basis of neighborhoods of zero in $s^\N$.

The family of unit vectors $\{e_{n,m}:n,m\in\N\}$, where
\[
  e_{n,m}=\left(\delta_{(n,m)(i,j)}\right)_{i,j=0}^\infty,
\]
is a Schauder basis for $s^\N$ (where $\delta$ is the Kronecker delta).
It follows from the definitions that
\[
  \|e_{n,m}\|_N
  = \begin{cases}
    A_{N,n}, & m\leq N;\\
    0,       & m>N.
  \end{cases}
\]
\subsection{The spaces \texorpdfstring{$s^\N_0$}{sᴺ₀} and \texorpdfstring{$s^\N_{0,0}$}{sᴺ₀₀}}
In our paper two dense linear subspaces of $s^\N$ will play an important role. We define
\[
  s^\N_0
  = \left\lbrace
    \left(x_{i,j}\right)_{i,j=0}^\infty\in s^\mathbb{N}:
      x_{i,j}=0
      \quad
      \text{for $j$ large enough}
  \right\rbrace
\]
and
\[
  s^\N_{0,0}
  = \left\lbrace
    \left(x_{i,j}\right)_{i,j=0}^\infty\in s^\mathbb{N}:
      x_{i,j}=0
      \quad
      \text{for all but finitely many $(i,j)$}
  \right\rbrace.
\]
The formula
\[
  \vertiii{\left(x_{i,j}\right)_{i,j=0}^\infty}_N
  = \sum_{j=0}^\infty
    \left| \left(x_{i,j}\right)_{\vphantom{j}i=0}^\infty \right|_N
\]
defines a norm on $s^\N_0$ and for every
$n,m,N\in N$ we have
\[
  \vertiii{e_{n,m}}_N = A_{N,n}.
\]

\section{The strategy}
\label{sec:stra}
In this section we describe the strategy for the construction of an operator
without non-trivial invariant subspaces on a separable Fréchet space $X$
with a Schauder basis $\{e_n:n\in\N\}$ and with the topology generated
by the fundamental family of seminorms $\{\|\cdot\|_n: n\in\N\}$ (we assume
that the unit balls of those seminorms form a basis of neighborhoods of zero in $X$).
This section is modeled on the last chapter of \cite{MR2533318} where F. Bayart
and \'{E}. Matheron gave a detailed exposition of Read's construction of
 an operator on $\ell_1$ without invariant subspaces.

The idea behind the construction is simple and is based on cyclic vectors.
\begin{definition} Let $T\colon X\to X$ be an operator. We say that
$x\in X$ is a cyclic vector for $T$ if
\[
  \Span \left\lbrace x,Tx,T^2x,\ldots\right\rbrace
  = \{
    P(T)x:
    \text{$P$ is a polynomial}
  \}
\]
is dense in $X$.
\end{definition}
It is a simple observation that an operator $T\colon X\to X$ has no non-trivial
invariant subspaces if and only if every non-zero $x\in X$ is a cyclic vector for $T$.
For some classes of operators the existence of at least one cyclic vector follows directly from the form of the operator.
\begin{definition} We say that an operator $T\colon X\to X$ is a perturbed weighted forward shift if for every $j\in \N$ we have
\[
  Te_j = \sum_{i=0}^{j+1} \lambda_ie_i,
\]
where $\lambda_{j+1}\not=0$.
\end{definition}
It is clear that $e_0$ is a cyclic vector for
any perturbed weighted forward shift.
We observe now that in order to show that any other non-zero vector
is cyclic for $T$ it is enough to prove an inequality.
\begin{proposition}
\label{prop:strategy}
Let $T\colon X\to X$ be an operator for which $e_0$ is a cyclic vector.
A non-zero $x\in X$ is a cyclic vector for
$T$ if and only if for every $N\in\N$ there is a polynomial $P$ such that
\begin{equation}
\label{eq:stra}
  \|P(T)x-e_0\|_N \leq 4.
\end{equation}
\end{proposition}
\begin{proof}
If $x$ is a cyclic vector for $T$, then
$\Span\left\lbrace x, Tx, T^2x,\ldots \right\rbrace$
is dense in $X$ and therefore for every $N\in \N$ there is a polynomial $P$ such that
\[
  \|P(T)x-e_0\|_N\leq 4.
\]
To prove the converse, recall that the unit balls of the family of seminorms
$\left\lbrace\|\cdot\|_n:n\in\N\right\rbrace$ form a basis of neighborhoods of zero in $X$.
Therefore the condition in the proposition implies that
\[
  e_0 \in \overline{\Span \left\lbrace x, Tx, T^2x,\ldots \right\rbrace}.
\]
But $e_0$ is a cyclic vector for $T$ and
\[
  \Span \left\lbrace e_0, Te_0, T^2e_0,\ldots \right\rbrace
  \subset
  \Span \left\lbrace x, Tx, T^2x,\ldots \right\rbrace.
\]
Thus $x$ is cyclic for $T$.
\end{proof}

In order to define an operator $T\colon X\to X$ for which \eqref{eq:stra}
is easy to verify for a large family of vectors we will use
the following lemma due to Read, which can be found in a slightly modified
version in the last chapter of \cite{MR2533318} (note that the proof
uses only simple linear algebra and a compactness argument).
\begin{definition}
We say that vectors $\gamma_0,\ldots,\gamma_n$
form a perturbed canonical basis of
$\Span\{e_0,\ldots,e_n\}$ if for every $0\leq j\leq n$
\[
  \gamma_j = \sum_{i=0}^j\mu_ie_i,
\]
where $\mu_j\not=0$.
\end{definition}
\begin{remark}
Let $T\colon X\to X$ be a perturbed weighted forward shift. For every $n\in N$
the vectors $e_0, Te_0,\ldots T^ne_0$ form a perturbed
canonical basis of $\Span\{e_0,\ldots,e_n\}$.
\end{remark}
\begin{lemma}\
\label{lemma:theLemma}
Let $\varepsilon>0$, $c$ and $d$ be positive integers and
$\gamma_0,\ldots, \gamma_{c+d-1}$ be a perturbed canonical basis
of $H=\Span\{e_0,\ldots,e_{c+d-1}\}$
such that
\[
  \gamma_c = \varepsilon e_c+e_0.
\]
Let $\|\cdot\|$ be a norm on $\Span \{e_n:n\in\mathbb{N}\}$ and
$K\subset H$ be a compact set such that
for every $y\in K$ with
\[
  y = \sum_{i=0}^{c+d-1}y_i\gamma_i
\]
there is $0\leq j\leq c-1$ such that $y_j \neq 0$.

Then there is a number $D\geq 1$ such that for every $y\in K$ and every
$T\colon X\to X$ which is a perturbed weighted forward shift such that
\[
  T^je_0=\gamma_j \text{ for } j=0,\ldots, c+d-1
\]
there is a polynomial
\[
  P(t) = \sum_{i=1}^{c+d}c_it^i \text{ with } \sum_{i=1}^{c+d}|c_i|\leq D
\]
for which
\[
  \|P(T)y-e_0\|
  \leq 2\varepsilon \|e_c\|
       + D\times \max_{c+d\leq j\leq 2(c+d-1)}\| T^je_0\|.
\]
\end{lemma}
This lemma will be used inductively
to construct an operator for which
\eqref{eq:stra} holds for many vectors from
$\Span\{e_n: n\in\N \}$.
Having an arbitrary vector $x=\sum_{n=0}^\infty x_ne_n$
we will write $x=h+t$ in such a way that
\eqref{eq:stra} holds for $h$ and $P(T)t$ is very small in the corresponding norm.
This in essence is what we will do.
\begin{remark}
In the proofs of Read, Goliński and Menet
the fact that $X$ possesses a continuous norm was crucial in the definition of the compact set $K$ when the above lemma is used. In the case of
$s^\N$ there is no continuous norm and we need to somehow overcome this difficulty.
\end{remark}
\begin{remark}
To realize the above strategy we need to define an order
on the canonical basis of $s^\N$.
This is done in the next section.
\end{remark}

\section{The definitions of the functions \texorpdfstring{$\Next$}{next} and \texorpdfstring{$\pos$}{pos}}
\begin{definition}
\label{def:order}
Let $(b_n)_{n=1}^\infty$ be an increasing sequence of natural numbers such that $2b_n+1<b_{n+1}$ for every $n\in\N$.
\begin{enumerate}[leftmargin=2\parindent]
\item We define the function
      \[
        \Next \colon \N\times\N \to \N\times\N
      \]
      by the formula (see also Figure \ref{fig:one}):
      \begin{equation}
      \label{def:next}
        \Next(i,j)=
        \begin{cases}
          (i,j+1),& i = b_n,\,    2 \mid j,\,      j < 2n;\\
          (i,j+1),& i = 2b_n+1,\, 2 \nmid j,\, 0 < j \leq 2n;\\
          (i,j+1),& i = 0,\,      2\nmid j ;\\
          (i,j-1),& i = b_n+1,\,  2\nmid j,\,      j < 2n;\\
          (i,j-1),& i = 2b_n,\,   2\mid j,\,   0 < j \leq 2n;\\
          \text{\parbox{\widthof{remaining cases}}{in the\\ remaining cases}}&
            \begin{cases}
              (i+1,j),& 2\mid j;\\
              (i-1,j),& 2\nmid j.\\
            \end{cases}
        \end{cases}
      \end{equation}
\item We define $\Next^0$ to be the identity on $\N\times\N$ and for $k\geq 1$ by
      $\Next^k$ we denote the composition of $\Next$ with itself $k$ times.
\item We define the function
      \[
        \pos \colon \N\times\N \to \N
      \]
      by the rule:
      \[
       \pos(n,m) = k
       \quad \text{if and only if} \quad
       \Next^k(0,0)=(n,m).
      \]
\item We define the sequence $(e_k)_{k=0}^\infty$ by the rule
      \[
        e_k = e_{n,m}
        \quad \text{if and only if} \quad
        \pos(n,m)=k.
      \]
\end{enumerate}
\end{definition}
\begin{remark}
  The function $\Next$ defines an ordering of $\N\times\N$ which is illustrated in Figure \ref{fig:one}.
  One should assume that the sequence $(b_n)_{n=0}^\infty$ increases very rapidly, so the true image would be very elongated.
\end{remark}

\begin{figure}[ht]
  \begin{minipage}{.5\textwidth}
  \centering
  \begin{tikzpicture}
    \coordinate (qx) at (0.5, 0);
    \coordinate (qy) at (0, -0.2);
    \coordinate (x0) at (0,0);
    \coordinate (x1) at ($1*(qx)$);
    \coordinate (x2) at ($2*(qx)$);
    \coordinate (x3) at ($3*(qx)$);
    \coordinate (x4) at ($4*(qx)$);
    \coordinate (x5) at ($5*(qx)$);
    \coordinate (x6) at ($6*(qx)$);
    \coordinate (x7) at ($7*(qx)$);
    \coordinate (x8) at ($8*(qx)$);
    \coordinate  (b0)  at (0, 0);

    \coordinate  (b1)  at (0, -0.6);
    \coordinate  (b1q) at ($(b1)+(qy)$);
    \coordinate (2b1)  at ($2*(b1)$);
    \coordinate (2b1q) at ($(2b1)+(qy)$);

    \coordinate  (b2)  at (0, -2);
    \coordinate  (b2q) at ($(b2)+(qy)$);
    \coordinate (2b2)  at ($2*(b2)$);
    \coordinate (2b2q) at ($(2b2)+(qy)$);

    \coordinate  (b3)  at (0, -5);
    \coordinate  (b3q) at ($(b3)+(qy)$);
    \coordinate (2b3)  at ($2*(b3)$);
    \coordinate (2b3q) at ($(2b3)+(qy)$);

    \draw[red] (0,0)
    |- (b1 -| x1)

    |- (b0 -| x2)
    |- (2b1 -| x1)
    |- (b1q -| x0)

    |- (b2 -| x1)
    |- (2b1q -| x2)
    |- (b2 -| x3)
    |- (b0 -| x4)
    |- (2b2 -| x3)
    |- (b2q -| x2)
    |- (2b2 -| x1)
    |- (b2q -| x0)

    |- (b3 -| x1)
    |- (2b2q -| x2)
    |- (b3 -| x3)
    |- (2b2q -| x4)
    |- (b3 -| x5)
    |- (b0 -| x6)
    |- (2b3 -| x5)
    |- (b3q -| x4)
    |- (2b3 -| x3)
    |- (b3q -| x2)
    |- (2b3 -| x1)
    |- (b3q -| x0)
    |- (2b3q);

    \foreach \x in {0,...,6}
    \foreach \y in {0,...,50}
    {
        \fill (0.5*\x,-0.2*\y) circle (1pt);
    }

    \node[left] at (b1) {$b_1$};
    \node[left] at (b2) {$b_2$};
    \node[left] at (b3) {$b_3$};

    \node[left] at (2b1) {$2b_1$};
    \node[left] at (2b2) {$2b_2$};
    \node[left] at (2b3) {$2b_3$};

    \node[above] at (x0) {$0$};
    \node[above] at (x1) {$1$};
    \node[above] at (x2) {$2$};
    \node[above] at (x3) {$3$};
    \node[above] at (x4) {$4$};
    \node[above] at (x5) {$5$};
    \node[above] at (x6) {$6$};
  \end{tikzpicture}
  \caption{Ordering defined by $\Next$}
  \label{fig:one}
  \end{minipage}%
  \begin{minipage}{.5\textwidth}
  \centering
  \begin{tikzpicture}
    \coordinate (qx) at (0.5, 0);
    \coordinate (qy) at (0, -0.2);
    \coordinate (x0) at (0,0);
    \coordinate (x1) at ($1*(qx)$);
    \coordinate (x2) at ($2*(qx)$);
    \coordinate (x3) at ($3*(qx)$);
    \coordinate (x4) at ($4*(qx)$);
    \coordinate (x5) at ($5*(qx)$);
    \coordinate (x6) at ($6*(qx)$);
    \coordinate (x7) at ($7*(qx)$);
    \coordinate (x8) at ($8*(qx)$);
    \coordinate (x9) at ($9*(qx)$);
    \coordinate (x10) at ($10*(qx)$);
    \coordinate  (b0)  at (0, 0);

    \coordinate  (b1)  at (0, -0.6);
    \coordinate  (b1q) at ($(b1)+(qy)$);
    \coordinate (2b1)  at ($2*(b1)$);
    \coordinate (2b1q) at ($(2b1)+(qy)$);

    \coordinate  (b2)  at (0, -1);
    \coordinate  (b2q) at ($(b2)+(qy)$);
    \coordinate (2b2)  at ($2*(b2)$);
    \coordinate (2b2q) at ($(2b2)+(qy)$);

    \coordinate  (D3)  at (0, -2.5);
    \coordinate  (b3)  at (0, -3.25);
    \coordinate  (b3q) at ($(b3)+(qy)$);
    \coordinate (2b3)  at ($2*(b3)$);
    \coordinate (2b3q) at ($(2b3)+(qy)$);

    \coordinate  (s3)  at (0, -7.5);
    \coordinate  (a3)  at (0, -9);
    \coordinate  (D4)  at (0, -10);

    \coordinate (x82) at ($0.5*(x8)$);
    \coordinate (b22) at ($0.5*(b2)$);

    \draw[pattern=north west lines] (0,0) -- (x8) -- (b2 -| x8) -- (b2) -- cycle;
    \node[fill=white, inner sep=1pt] at (x82 |- b22) {previous intervals};

    \draw (b2) 
    -- (2b2);
    \draw[red] (2b2)
    |- (b3 -| x1)
    |- (b2q -| x2)
    |- (b3 -| x3)
    |- (b2q -| x4)
    |- (b3 -| x5)
    |- (b2q -| x6)
    |- (b3 -| x7)
    |- (b2q -| x8)
    |- (b3 -| x9)
    |- (b0 -| x10)
    |- (2b3 -| x9)
    |- (b3q -| x8)
    |- (2b3 -| x7)
    |- (b3q -| x6)
    |- (2b3 -| x5)
    |- (b3q -| x4)
    |- (2b3 -| x3)
    |- (b3q -| x2)
    |- (2b3 -| x1)
    |- (b3q -| x0)
    -- (s3)
    -- (a3)
    -- (D4);

    \draw (D4)
    -- (0,-10.2);

    \node[left] at (b3) {$b_n$};
    \draw ($(b3)-(0.1,0)$) -- ($(b3)+(0.1,0)$);
  
    \node[left] at (2b2) {$\Delta_n$};
    \draw ($(2b2)-(0.1,0)$) -- ($(2b2)+(0.1,0)$);
    \node[left] at (2b3) {$2b_n$};
    \draw ($(2b3)-(0.1,0)$) -- ($(2b3)+(0.1,0)$);

    \node[left] at (s3) {$s_n$};
    \draw ($(s3)-(0.1,0)$) -- ($(s3)+(0.1,0)$);

    \node[left] at (a3) {$a_n$};
    \draw ($(a3)-(0.1,0)$) -- ($(a3)+(0.1,0)$);

    \node[left] at (D4) {$\Delta_{n+1}$};
    \draw ($(D4)-(0.1,0)$) -- ($(D4)+(0.1,0)$);

    \node[above] at (x0) {$0$};
    \node[above] at (x1) {$1$};
    \node[above] at (x2) {$2$};
    \node[above] at (x3) {$\ldots$};
    \node[above] at (x4) {};
    \node[above] at (x5) {};
    \node[above] at (x6) {};
    \node[above] at (x8) {$2n-2$};
    \node[above] at (x10) {$2n$};
  \end{tikzpicture}
  \caption{A single step of the inductive procedure}
  \label{fig:two}
  \end{minipage}%
\end{figure}
\begin{remark}
  It is clear that the elements of the sequence $(e_n)_{n=0}^\infty$ form a Schauder basis of $s^\N$.
  It is also easy to see that
  \[
    e_{\pos(n,m)} = e_{n,m}.
  \]
\end{remark}

\section{Construction of the operator}
We fix a sequence
\begin{equation*}
  (N_n)_{n\in\N } = (0, 1, 0, 1, 2, 0, 1, 2, 3, 0, 1, 2, 3, 4, \ldots).
\end{equation*}
Two properties of this sequence will be important for us: every non-negative integer appears in this sequence infinitely many times and
\begin{equation}
\label{eq:N_n}
  N_n \leq n \quad \text{for every $n\in N$}.
\end{equation}
Assume that we are given integer sequences $(a_n)_{n=0}^\infty$, $(b_n)_{n=1}^\infty$, $(s_n)_{n=1}^\infty$ satisfying
\begin{equation}
  \begin{aligned}
    \label{eq:monotonicity}
    1 & = \Delta_0 \phantom{\mbox{} < b_1 < 2b_1 < s_1}\mbox{} < a_0 < a_0 + \pos(\Delta_0, 0) \\
      & = \Delta_1 < b_1 < 2b_1 < s_1  < a_1    < a_1 + \pos(\Delta_1, 0) \\
      & = \Delta_2 < b_2 < 2b_2 < s_2  < a_2    < a_2 + \pos(\Delta_2, 0) \\
      & = \Delta_3 \ldots,
  \end{aligned}
\end{equation}
where the function $\pos$ is derived from the sequence $(b_n)_{n=1}^\infty$ as described in the previous section.

For further reference, note that
from the definition of $\Next$ and $\pos$ it follows that
\begin{equation}
\label{eq:pos}
  \pos(\Delta_{n+1},0) = \pos(a_n,0) + \pos(\Delta_n,0).
\end{equation}
Moreover, assume we are given a sequence $(D_n)_{n=0}^\infty$ of non-zero numbers. We define a linear operator
\[
  T\colon s^\N_{0,0}\to s^\N_{0,0}
\]
by the formula
\begin{equation}
\label{eq:operator}
  T^je_0
  = \begin{cases}
    \alpha_j e_j,                                  & j\in \left[\pos(\Delta_n,0),\pos(a_n,0)\right);\\
    \frac{1}{A_{N_n,a_n}}e_j+T^{j-\pos(a_n,0)}e_0, & j\in \left[\pos(a_n,0),\pos(\Delta_{n+1},0)\right);
  \end{cases}
\end{equation}
where
\begin{equation}
\label{eq:alphas}
  \alpha_j
  = \begin{cases}
    \frac{1}{A_{N_0,a_0}} \cdot 2^{j-\pos(\Delta_0,0)},                                & j\in [\pos(\Delta_0, 0),\pos(a_0,0)); \\
    \frac{1}{A_{N_n,a_n}} \cdot \left(\frac{1}{2D_{n-1}}\right)^{j- \pos(\Delta_n,0)}, & j\in [\pos(\Delta_n,0),\pos(s_n,0)), n \geq 1;\\
    \alpha_{\pos(s_n,0)-1}\cdot 2^{j-\pos(s_n,0)},                                     & j\in [\pos(s_n, 0),\pos(a_n,0)), n \geq 1.
  \end{cases}
\end{equation}
\begin{remark}
Observe that because the function $\Next$ (defined with $(b_n)_{n=1}^\infty$ by \eqref{def:next}) gives us a strict ordering of $\N\times\N$ and all the numbers $\alpha_j$ in \eqref{eq:operator} are non-zero, one can easily deduce that \eqref{eq:operator} really defines a linear operator on $s^\N_{0,0}$ (we calculate the values of $Te_j$ in Section 7).
\end{remark}

\section{Choosing the parameters}
\label{sec:parameters}
The sequences $(a_n)_{n=0}^\infty$, $(b_n)_{n=1}^\infty$, $(s_n)_{n=1}^\infty$ that were used in the previous section will be carefully chosen in an inductive procedure carried out over the intervals $[\Delta_n, \Delta_{n+1})$ described in this section. Along the way we will also fix the corresponding sequence $(D_n)_{n=0}^\infty$.

The procedure on the first interval starting in $\Delta_0$ is very similar to the general step, but simplified, as there is no need for $b_0$ and $s_0$. Hence we will go over the first step of the induction as well as the general step at the same time.

Assume that $(a_k)_{k=0}^{n-1}$, $(b_k)_{k=1}^{n-1}$, $(s_k)_{k=1}^{n-1}$ and $(D_k)_{k=0}^{n-1}$ have been already fixed. Hence the values of $T^je_0$ are defined by \eqref{eq:operator} for $j$ up to (but not including) $\pos(a_{n-1},0)+\pos(\Delta_{n-1},0)$. We define $\Delta_n=a_{n-1}+\pos(\Delta_{n-1},0)$ in accordance with \eqref{eq:monotonicity} (with $\Delta_0=1$).

If $n \geq 1$, then we choose $b_n$ to be an integer such that $b_n \geq 2 \pos(\Delta_n,0)$. Hence:
\begin{equation}
  \label{cond:pos_bn}
  2(\pos(\Delta_n,0)-1)\leq b_n \leq \pos(b_n,0).
\end{equation}
We also require that for $k\geq b_n$ we have that
\begin{equation}
  \label{cond:2bn}
  4^{\pos(\Delta_n,0)} A_{N_{n-1}+1,k}
  \leq
  \frac{A_{N_{n-1}+2,k}}{D_{n-1}}.
\end{equation}
This is possible as $\frac{A_{N_{n-1}+2,k}}{A_{N_{n-1}+1,k}}$ tends to $+\infty$ (see Proposition \ref{prop:matrix}).

In $(b_n,0)$ the order defined by $\Next$ takes us to other columns than the first (see Figure \ref{fig:two}). The number of iterates we need to return to ($b_{n+1,0}$) is determined by $\{b_1, b_2, \ldots, b_n\}$ only, hence we can calculate $\pos(x,0)$ for $x \leq b_{n+1}$, even though $b_{n+1}$ is not yet fixed and will be chosen in the next step. We choose $s_n$ to be an integer bigger than $2b_n$. The number $b_0$ is in fact not used in the definition of $\Next$, for simplicity we may choose $b_0=s_0=\Delta_0=1$.

We can now choose $a_n$ to be an integer greater than $s_n$ satisfying:
\begin{equation}
\label{eq:cond1}
  \frac{2^{\pos(a_n,0)-\pos(s_n,0)-1}}{A_{N_n,a_n} \left( 2 D_{n-1} \right)^{\pos(s_n,0)-\pos(\Delta_n,0)-1}} \geq 1;
\end{equation}
\begin{equation}
\label{eq:cond3}
  A_{N_n,a_n} 4^{\pos(\Delta_n,0)}A_{0,0}
  \leq
  A_{N_n+1,a_n}.
\end{equation}
This is possible by Proposition \ref{prop:matrix}. For $n \geq 1$ we also require that:
\begin{equation}
\label{eq:cond2}
  \frac{A_{N_{n-1},a_{n-1}}}{A_{N_{n},a_{n}}} \leq 1;
\end{equation}

\begin{equation}
\label{eq:cond4}
  \frac{D_{n-1}A_{N_{n-1},b_{n}}}{A_{N_{n}},a_{n}} \leq 1.
\end{equation}
Now the equation \eqref{eq:alphas} gives as the values for $\alpha_j$ for $j \in [\pos(\Delta_n,0), \pos(a_n,0))$. Hence the values of $T^j e_0$ are defined for $j < \pos(a_n,0) + \pos(\Delta_n,0)$ by \eqref{eq:operator}. We put $\Delta_{n+1} = a_n + \pos(\Delta_n,0)$. For further reference note that \eqref{eq:cond1} gives us immediately

\begin{equation}
  \label{eq:alpha_a_n}
  \alpha_{\pos(a_n,0)-1} \geq 1.
\end{equation}

To carry out the construction on the next interval we still need a suitable number $D_n$. It will be an outcome of applying Lemma \ref{lemma:theLemma}.

Let
\begin{equation}
\label{eq:heads}
  H_n
  = \Span\left\lbrace e_j:
    j\leq \pos(\Delta_{n+1},0)-1
  \right\rbrace.
\end{equation}
Then $\{T^je_0: j\leq \pos(\Delta_{n+1},0)-1\}$ is a perturbed canonical basis of $H_n$.
We define a linear projection $\tau_n\colon H_n\to H_n$ by the formula
\begin{equation}
\label{eq:tau_n}
  \tau_n\left(\sum_{j=0}^{\pos(\Delta_{n+1},0)-1} x_j T^je_0\right)
  = \sum_{j=0}^{\pos(a_n,0)-1} x_jT^je_0.
\end{equation}
Using this projection we define a compact set
\begin{equation}
\label{eq:compacts}
  K_n
  = \left\lbrace
    y\in H_n: \vertiii{y}_0\leq 1 \textrm{ and }
    \vertiii{\tau_n(y)}_0\geq \frac{1}{2}
  \right\rbrace.
\end{equation}

We now apply Lemma \ref{lemma:theLemma} with
$H=H_n$, $K=K_n$, $\|\cdot\|=\vertiii{\cdot}_{N_n}$, $c=\pos(a_n,0)$ and $d=\pos(\Delta_n,0)$. From this lemma we get the necessary number $D_n$ and we can carry out the procedure on the next interval.

Applying Lemma \ref{lemma:theLemma} in the construction gives us the following:
\begin{proposition}
  \label{prop:heads}
  If $y\in K_n$, then there is a polynomial
  \begin{equation}
  \label{eq:polynomial}
    P(t)=\sum_{i=1}^{\pos(\Delta_{n+1},0)}c_it^i
    \text{ with }
    \sum_{i=1}^{\pos(\Delta_{n+1},0)}|c_i|\leq D_n
  \end{equation}
  for which
  \[
    \vertiii{P(T)y-e_0}_{N_n} \leq 3.
  \]
\end{proposition}
\begin{proof}
Let $y\in K_n$.
Lemma \ref{lemma:theLemma} gives us the existence of a polynomial $P$
satisfying \eqref{eq:polynomial} such that
\begin{align*}
  \vertiii{P(T)y-e_0}_{N_n}
  & \leq
    \frac{2}{A_{N_n,a_n}} \vertiii{e_{\pos(a_n,0)}}_{N_n}
    + D_n \times \max_{\pos(\Delta_{n+1},0)\leq j\leq 2(\pos(\Delta_{n+1},0)-1)}
                 \vertiii{T^je_0}_{N_n}\\
  & \overset{\mathclap{\eqref{eq:operator}}}{=}
    2 + D_n \times \max_{\pos(\Delta_{n+1},0)\leq j\leq 2(\pos(\Delta_{n+1},0)-1)}
            \vertiii{\alpha_je_j}_{N_n}\\
  & \overset{\mathclap{\eqref{eq:alphas}}}{\leq}
    2 + D_n \alpha_{\pos(\Delta_{n+1},0)} \vertiii{e_{2\pos(\Delta_{n+1},0)}}_{N_n}\\
  &\overset{\mathclap{\eqref{eq:alphas}}}{\leq}
    2 + \frac{D_n}{A_{N_{n+1}},a_{n+1}} \vertiii{e_{\pos(b_{n+1},0)}}_{N_n}\\
  & =
    2 + \frac{D_nA_{N_{n},b_{n+1}}}{A_{N_{n+1}},a_{n+1}}
  \overset{\mathclap{\eqref{eq:cond4}}}{\leq} 3.\qedhere
\end{align*}
\end{proof}

\begin{corollary}
\label{cor:heads}
If $\displaystyle y\in \bigcup_{m=1}^\infty mK_n$,
then there is a polynomial
\begin{equation*}
  P(t) = \sum_{i=1}^{\pos(\Delta_{n+1},0)} c_it^i
  \text{ with }
  \sum_{i=1}^{\pos(\Delta_{n+1},0)}|c_i| \leq D_n
\end{equation*}
for which
\[
  \vertiii{P(T)y-e_0}_{N_n}\leq 3.
\]
\end{corollary}
\begin{proof} By the assumptions there is $m\geq 1$ such that
$
  y' = \frac{1}{m} y\in K_n.
$
By Proposition \ref{prop:heads}
there is a polynomial
\begin{equation*}
  P(t) = \sum_{i=1}^{\pos(\Delta_{n+1},0)}c_it^i
  \text{ with }
  \sum_{i=1}^{\pos(\Delta_{n+1},0)}|c_i| \leq D_n
\end{equation*}
for which
\[
  \vertiii{P(T)y'-e_0}_{N_n}\leq 3.
\]
To finish the proof observe that the polynomial
$P' = \frac{1}{m}P$
has all the requested properties.
\end{proof}

\section{The values \texorpdfstring{$Te_j$}{Teⱼ}}
To prove that the operator $T$ has certain properties we need to calculate the exact values of $Te_j$.

\begin{proposition}
\label{prop:values}
The operator $T\colon s_{0,0}^\N\to s_{0,0}^\N$
defined by \eqref{eq:operator} satisfies the equality:
\[
  Te_j
  = \begin{cases}
    \alpha_1 e_1,                         & j = 0;\\
    \frac{\alpha_{j+1}}{\alpha_j}e_{j+1}, & j \in \left[\pos(\Delta_n,0), \pos(a_n,0)-1\right);\\
    \frac{1}{\alpha_{\pos(a_n,0)-1}}
      \left(\frac{1}{A_{N_n,a_n}} e_{\pos(a_n,0)} +e_0\right),
                                          & j = \pos(a_n,0)-1;\\
     e_{j+1},                             & j \in \left[\pos(a_n,0),\pos(\Delta_{n+1},0)-1\right);\\
     A_{N_n,a_n}\alpha_{\pos(\Delta_{n+1},0)} e_{\pos(\Delta_{n+1},0)}
       - A_{N_n,a_n}\alpha_{\pos(\Delta_n,0)} e_{\pos(\Delta_n,0)},
                                          & j = \pos(\Delta_{n+1}, 0)-1.
  \end{cases}
\]
Thus $T$ is a perturbed weighted forward shift.
\end{proposition}
\begin{proof}
We consider all the possible cases for $j$.
All the calculations below follow from \eqref{eq:operator}.
\begin{enumerate}[leftmargin=2\parindent]
\item By the definition we have $Te_0=\alpha_1e_1$.
\item If $j \in \left[\pos(\Delta_n,0),\pos(a_n,0)-1\right)$, then
      \[
        e_j
        = \frac{1}{\alpha_j} T^je_o.
      \]
      Therefore
      \[
        Te_j
        = \frac{1}{\alpha_j}T^{j+1}e_0=\frac{\alpha_{j+1}}{\alpha_j}e_{j+1}.
      \]
\item If $j = \pos(a_n,0)-1$, then
      \[
        e_{\pos(a_n,0)-1}
        = \frac{1}{\alpha_{\pos(a_n,0)-1}}T^{\pos(a_n,0)-1}e_0.
      \]
      Therefore
      \[
        Te_{\pos(a_n,0)-1}
        = \frac{1}{\alpha_{\pos(a_n,0)-1}}T^{\pos(a_n,0)}e_0
        = \frac{1}{\alpha_{\pos(a_n,0)-1}}\left(\frac{1}{A_{{N_n},a_n}}e_{\pos(a_n,0)} + e_0\right).
      \]
\item If $j \in \left[\pos(a_n,0),\pos(\Delta_{n+1},0)-1\right)$, then
      \[
        e_j = A_{N_n,a_n}\left(T^je_0-T^{j-\pos(a_n,0)}e_0 \right).
      \]
      Therefore
      \begin{align*}
        Te_j
        & = A_{N_n,a_n}\left(T^{j+1}e_0-T^{j+1-\pos(a_n,0)}e_0 \right)\\
        & = A_{N_n,a_n}\left(\frac{1}{A_{N_n,a_n}}e_{j+1}
            + T^{j+1-\pos(a_n,0)}e_0-T^{j+1-\pos(a_n,0)}e_0\right)\\
        & = e_{j+1}.
      \end{align*}
\item If $j = \pos(\Delta_{n+1},0)-1$, then
      \[
        e_{\pos(\Delta_{n+1},0)-1}
        = A_{N_n,a_n}\left(T^{\pos(\Delta_{n+1},0)-1}e_0-T^{\pos(\Delta_{n+1},0)-1-\pos(a_n,0)}e_0 \right).
      \]
      Using \eqref{eq:pos} we get
      \begin{align*}
        Te_{\pos(\Delta_{n+1},0)-1}
        & = A_{N_n,a_n}\left(T^{\pos(\Delta_{n+1},0)} e_0-T^{\pos(\Delta_n,0)}e_0 \right)\\
        & = A_{N_n,a_n}\alpha_{\pos(\Delta_{n+1},0)} e_{\pos(\Delta_{n+1},0)}
            - A_{N_n,a_n}\alpha_{\pos(\Delta_n,0)}e_{\pos(\Delta_n,0)}.\qedhere
      \end{align*}
\end{enumerate}
\end{proof}

\section{Continuity of \texorpdfstring{$T$}{T}}
In this section we show that the linear
operator $T$ defined on the space $s^\N_{0,0}$
can be extended to a continuous linear
operator acting from $s^\N$ to $s^\N$.
We also prove some continuity type
estimates which will be
important in the forthcoming sections.

\begin{lemma}
\label{lemma:estimate}
For every $j\in \mathbb{N}$
and $N\in\mathbb{N}$
we have $\|e_{j+1}\|_N \leq 2\|e_j\|_{N+1}$ and $\vertiii{e_{j+1}}_N\leq 2 \vertiii{e_j}_N$.
\end{lemma}
\begin{proof}
This follows directly from the definition of the vectors
$e_j$ and the definition of
$\|\cdot\|_N$ and $\vvvert \cdot \vvvert_N$.
\end{proof}

\begin{proposition}
\label{prop:continuity}
The operator
$T\colon s_{0,0}^\N\to s_{0,0}^\N$
defined by \eqref{eq:operator}
satisfies the inequality
$\|Tx\|_N\leq 4\|x\|_{N+1}$
for every $x\in s_{0,0}^\N$
and $N\in\mathbb{N}$.

In particular, $T$ can be extended to a
continuous operator on $s^\N$.
Moreover, the extension of $T$
(which we also denote by $T$)
satisfies
the inequality
$\vertiii{Tx}_N\leq 4\vertiii{x}_N$
for every $x\in s_0^\N$
and $N\in \N$.
\end{proposition}
\begin{proof}
Fix $N\in\N$. First we will show that for
every $(n,m)\in \N\times\N$ we have
\begin{equation}
\label{eq:cont}
  \left\|Te_{n,m}\right\|_N\leq 4\left\|e_{n,m}\right\|_{N+1}
  \quad\textrm{and}\quad
  \vertiii{Te_{n,m}}_N\leq 4\vertiii{e_{n,m}}_N.
\end{equation}
Let $j\in\N$ be the unique number such that
\[
  e_j=e_{n,m}.
\]
We consider all the possible cases.
\begin{enumerate}[leftmargin=2\parindent]
\item If $j=0$, then by Proposition \ref{prop:values}
      \[
        Te_0 = \alpha_1e_1.
      \]
      Hence, using Lemma \ref{lemma:estimate}
      \[
        \|Te_0\|_N
        = \vertiii{Te_0}_N
        = \vertiii{\alpha_1e_1}_N
        \leq 2 \vertiii{e_0}_N
        = 2 \|e_0\|_N.
      \]
\item If $j \in \left[\pos(\Delta_n,0),\pos(a_n,0)-1\right)$,
      then by Proposition \ref{prop:values}
      \[
        Te_j = \frac{\alpha_{j+1}}{\alpha_j}e_{j+1}.
      \]
      By the definition of the numbers $\alpha_j$ (see \eqref{eq:alphas}) we have
      \[
        \frac{\alpha_{j+1}}{\alpha_j} \leq 2.
      \]
      Therefore, using Lemma \ref{lemma:estimate}, we get
      \[
        \left\|Te_j\right\|_N
          = \left\|\frac{\alpha_{j+1}}{\alpha_j}e_{j+1}\right\|_N
          \leq 4 \left\|e_j\right\|_{N+1}
        \quad\textrm{and}\quad
        \vertiii{Te_j}_N
          = \vertiii{\frac{\alpha_{j+1}}{\alpha_j}e_{j+1}}_N
          \leq 4 \vertiii{e_j}_N.
      \]

\item If $j=\pos(a_n,0)-1$, then by Proposition \ref{prop:values}
      \[
        Te_{\pos(a_n,0)-1}
        = \frac{1}{\alpha_{\pos(a_n,0)-1}}
          \left(\frac{1}{A_{N_n,a_n}} e_{\pos(a_n,0)} +e_0\right).
      \]
      We have
      \begin{align*}
        \left\|
          \frac{1}{\alpha_{\pos(a_n,0)-1}}\left(\frac{1}{A_{N_n},a_n}e_{\pos(a_n,0)} + e_0\right)
        \right\|_N
        & = \vertiii{\frac{1}{\alpha_{\pos(a_n,0)-1}}\left(\frac{1}{A_{N_n,a_n}} e_{\pos(a_n,0)} + e_0\right)}_N \\
        & = \frac{\frac{A_{N,a_n}}{A_{N_n,a_n}}+A_{N,0}}{\alpha_{\pos(a_n,0)-1}}
      \end{align*}
      and
      \[
        \left\| e_{\pos(a_n,0)-1}\right\|_N
          = \vertiii{e_{\pos(a_n,0)-1}}_N
          = A_{N,a_n-1}.
      \]
      Moreover, from \eqref{eq:alpha_a_n} and Proposition \ref{prop:matrix}
      \[
        \frac{\frac{A_{N,a_n}}{A_{N_n,a_n}} + A_{N,0}}{\alpha_{\pos(a_n,0)-1}}
        \leq 4 A_{N,a_n-1}.
      \]
      
      This shows that
      \[
        \left\|T e_{\pos(a_n,0)-1}\right\|_N
        = \vertiii{Te_{\pos(a_n,0)-1}}_N
        \leq 4\left\| e_{\pos(a_n,0)-1}\right\|_N
        = 4 \vertiii{e_{\pos(a_n,0)-1}}_N.
      \]
\item If $j\in \left[\pos(a_n,0),\pos(\Delta_{n+1},0)-1\right)$,
      then by Proposition \ref{prop:values}
      \[
        Te_j = e_{j+1}.
      \]
      Therefore, using Lemma \ref{lemma:estimate}, we get
      \[
        \left\|Te_j\right\|_N
        = \left\|e_{j+1}\right\|_N\leq 4\left\|e_j\right\|_{N+1}
        \quad\textrm{and}\quad
        \vertiii{Te_j}_N
        = \vertiii{e_{j+1}}_N\leq 4\vertiii{e_j}_N.
      \]
\item If $j=\pos(\Delta_{n+1},0)-1$, then by Proposition \ref{prop:values}
      \[
        Te_{\pos(\Delta_{n+1},0)-1}
        = A_{N_n,a_n} \alpha_{\pos(\Delta_{n+1},0)} e_{\pos(\Delta_{n+1},0)}
          - A_{N_n,a_n} \alpha_{\pos(\Delta_n,0)} e_{\pos(\Delta_n,0)}.
      \]
      We have
      \begin{align*}
      & \left\|A_{N_n,a_n} \alpha_{\pos(\Delta_{n+1},0)} e_{\pos(\Delta_{n+1},0)}
          - A_{N_n,a_n}\alpha_{\pos(\Delta_n,0)}e_{\pos(\Delta_n,0)} \right\|_N\\
      & \hspace{2cm} =
        \vertiii{A_{N_n,a_n} \alpha_{\pos(\Delta_{n+1},0)} e_{\pos(\Delta_{n+1},0)}
        - A_{N_n,a_n} \alpha_{\pos(\Delta_n,0)} e_{\pos(\Delta_n,0)}}_N\\
      & \hspace{2cm} =
        A_{N_n,a_n} \alpha_{\pos(\Delta_{n+1},0)} A_{N,\Delta_{n+1}}
        + A_{N_n,a_n} \alpha_{\pos(\Delta_n,0)} A_{N,\Delta_n}\\
      & \hspace{2cm} \overset{\mathclap{\eqref{eq:alphas}}}{=}
        \frac{A_{N_n,a_n} A_{N,\Delta_{n+1}}}{A_{N_{n+1},a_{n+1}}}
        + \frac{A_{N_n,a_n}A_{N,\Delta_n}}{A_{N_n,a_n}}
      \end{align*}
      and
      \[
        \left\|e_{\pos(\Delta_{n+1},0)-1}\right\|_N
        = \vertiii{e_{\pos(\Delta_{n+1},0)-1}}_N=A_{N,\Delta_{n+1}-1}.
      \]
      By \eqref{eq:cond2} we have
      \[
        \frac{A_{N_n,a_n}}{A_{N_{n+1},a_{n+1}}} \leq 1.
      \]
      Therefore, since
      $A_{N,\Delta_{n+1}}\leq 2A_{N,\Delta_{n+1}-1}$ and
      $A_{N,\Delta_{n}}\leq 2A_{N,\Delta_{n+1}-1}$ (see Proposition \ref{prop:matrix}),
      we get that
      \[
        \frac{A_{N_n,a_n}A_{N,\Delta_{n+1}}}{A_{N_{n+1},a_{n+1}}}
          + \frac{A_{N_n,a_n}A_{N,\Delta_n}}{A_{N_n,a_n}}
        \leq  4 A_{N,\Delta_{n+1}-1}.
      \]
      This shows that
      \[
        \left\|Te_{\pos(\Delta_{n+1},0)-1}\right\|_N
        = \vertiii{Te_{\pos(\Delta_{n+1},0)-1}}_N
        \leq \left\|e_{\pos(\Delta_{n+1},0)-1}\right\|_N
        = \vertiii{e_{\pos(\Delta_{n+1},0)-1}}_N.
      \]
\end{enumerate}

After all the tedious calculations we are now ready to prove the desired inequalities.
Let $\left(x_{i,j}\right)_{i,j=0}^\infty\in s_{0,0}^\N$ be arbitrary.
Then
\begin{align*}
  \left\|T\left(x_{i,j}\right)_{i,j=0}^\infty\right\|_N
  = \left\|\sum_{i,j=0}^\infty x_{i,j}Te_{i,j}\right\|_N
  & \leq
    \sum_{i,j=0}^\infty |x_{i,j}|\left\| Te_{i,j} \right\|_N
  \overset{\eqref{eq:cont}}{\leq}
    4 \sum_{i,j=0}^\infty |x_{i,j}|\left\| e_{i,j} \right\|_{N+1}\\
  & =
    4 \sum_{j=0}^{N+1} \sum_{i=0}^\infty |x_{i,j}|\left\|e_{i,j} \right\|_{N+1}
  =
  4 \sum_{j=0}^{N+1} \left| \left(x_{i,j}\right)_{\vphantom{j}i=0}^\infty \right|_{N+1}\\
  & =
  4\left\|\left(x_{i,j}\right)_{i,j=0}^\infty\right\|_{N+1}.
\end{align*}
Since $s_{0,0}^\N$ is dense in $s^\N$,
this implies that $T$ can be extended
to a continuous operator on
$s^\N$. We denote this extension by $T$.

Let now $x=\left(x_{i,j}\right)_{i,j=0}^\infty\in s_0^\N$.
There is $M\in\mathbb{N}$ such that
\[
  x = \sum_{j=0}^M \sum_{i=0}^\infty x_{i,j}e_{i,j}.
\]
It is clear from Proposition \ref{prop:values} that $Tx\in s_0^\N$.
Moreover, from \eqref{eq:cont} we get
\begin{align*}
  \vertiii{T\left(x_{i,j}\right)_{i,j=0}^\infty}_N
  & = 
    \vertiii{\sum_{j=0}^M\sum_{i=0}^\infty x_{i,j}Te_{i,j}}_N
  \leq
    4 \sum_{j=0}^M \sum_{j=0}^\infty \left|x_{i,j}\right|\vertiii{e_{i,j}}_N
  =
    4 \vertiii{\sum_{j=0}^M \sum_{i=0}^\infty x_{i,j}e_{i,j}}_N.\qedhere
\end{align*}
\end{proof}

\section{The properties of the projection \texorpdfstring{$\tau_n$}{τₙ}}
In this section we investigate the properties of the projection
$\tau_n$ defined in Section \ref{sec:parameters}.

\begin{lemma}
\label{lemma:values_tau}
The projection $\tau_n\colon H_n\to H_n$ defined by \eqref{eq:tau_n} satisfies the equation
\begin{equation*}
\tau_n(e_j)
  = \begin{cases}
    e_j,                              & j\leq \pos(a_n,0)-1;\\
    -A_{N_n,a_n}T^{j-\pos(a_n,0)}e_0, & \pos(a_n,0)\leq j \leq \pos(\Delta_{n+1},0)-1.
  \end{cases}
\end{equation*}
\end{lemma}
\begin{proof}
Let $j\leq \pos(\Delta_{n+1},0)-1$.
The vectors
$e_0, Te_0,\ldots T^je_0$
form a basis of $\Span\{e_0,\ldots,e_j\}$
and therefore
\[
  e_j = \sum_{i=0}^j \lambda_i T^i e_0
\]
for some numbers $\lambda_0,\ldots,\lambda_j$.

\begin{enumerate}[leftmargin=2\parindent]
\item If $j\leq \pos(a_n,0)-1$, then from \eqref{eq:tau_n} it follows that
      \[
        \tau_n \left(e_j\right)
          = \tau_n \left(\sum_{i=0}^j\lambda_iT^ie_0 \right)
          = \sum_{i=0}^j\lambda_iT^ie_0
          = e_j.
      \]
\item If $\pos(a_n,0)\leq j \leq \pos(\Delta_{n+1},0)-1$,
then from the definition of the operator $T$ (see \eqref{eq:operator})
it follows that
\[
  e_j=A_{N_n,a_n}\left(T^je_0-T^{j-\pos(a_n,0)}e_0\right).
\]
By \eqref{eq:pos} we have that
\[
 j-\pos(a_n,0) \leq \pos(\Delta_n,0) \leq \pos(a_n,0)-1.
\]
This and \eqref{eq:tau_n} gives that
\[
  \tau_n \left(e_j\right)
  = \tau_n \left(A_{N_n,a_n} \left(T^je_0-T^{j-\pos(a_n,0)} e_0\right)\right)
  = -A_{N_n,a_n} T^{j-\pos(a_n,0)}e_0. \qedhere
\]
\end{enumerate}
\end{proof}

The next proposition gives a continuity type estimate for the projection
$\tau_n$ which will be important for us in the next section.
\begin{proposition}
  \label{prop:cont_tau}
  Let $n\in N$ and $x\in H_n$. Then $\vertiii{\tau_n x}_0\leq \vertiii{x}_{N_n+1}$.
\end{proposition}
\begin{proof}
First we will show that for $j\leq \pos(\Delta_{n+1},0)-1$
we have $\vertiii{\tau_ne_j}_0\leq \vertiii{e_j}_{N_n+1}$.
To do this we consider all possible cases.
\begin{enumerate}[leftmargin=2\parindent]
\item If $j\leq \pos(a_n,0)-1$, then from Lemma \ref{lemma:values_tau}
      it follows that
      \[
        \tau_n(e_j) = e_j.
      \]
      Thus
      \[
        \vertiii{\tau_n(e_j)}_0 = \vertiii{e_j}_0\leq \vertiii{e_j}_{N_n+1}.
      \]
\item If $\pos(a_n,0)\leq j \leq \pos(\Delta_{n+1},0)-1$,
      then from Lemma \ref{lemma:values_tau} it follows that
      \[
        \tau_n(e_j) = -A_{N_n,a_n} T^{j-\pos(a_n,0)} e_0.
      \]
      Using Proposition \ref{prop:continuity} we get that
      \begin{align*}
        \vertiii{\tau_n(e_j)}_0
        = \vertiii{A_{N_n,a_n}T^{j-\pos(a_n,0)}e_0}_0
        \leq A_{N_n,a_n} 4^{j-\pos(a_n,0)} \vertiii{e_0}_0
        = A_{N_n,a_n}4^{j-\pos(a_n,0)}A_{0,0}.
      \end{align*}
      By \eqref{eq:pos} we have
      \[
        j-\pos(a_n,0) \leq \pos(\Delta_n,0),
      \]
      Therefore
      \begin{equation}
        \label{eq:1}
        \vertiii{\tau_n(e_j)}_0
        \leq A_{N_n,a_n} 4^{\pos(\Delta_n,0)}A_{0,0}.
      \end{equation}
      Since $\pos(a_n,0) \leq j \leq \pos(\Delta_{n+1},0)-1$, we get that
      \begin{equation}
        \label{eq:2}
        \vertiii{e_j}_{N_n+1}
        \geq \vertiii{e_{\pos(a_n,0)}}_{N_n+1}
        = \vertiii{e_{a_n,0}}_{N_n+1}=A_{N_n+1,a_n}.
      \end{equation}
      By \eqref{eq:cond3} we have that
      \[
        A_{N_n,a_n} 4^{\pos(\Delta_n,0)} A_{0,0}
        \leq A_{N_n+1,a_n}.
      \]
      Thus, putting now \eqref{eq:1} and \eqref{eq:2} together, we get
      \[
        \vertiii{\tau_n(e_j)}_0
        \leq \vertiii{e_j}_{N_n+1}.
      \]
\end{enumerate}

To finish the proof, take $x\in H_n$. Then
\[
  x = \sum_{j=0}^{\pos(\Delta_{n+1},0)-1} x_j e_j
\]
for some numbers $x_j$.
By the inequality which we have just proved we get
\[
  \vertiii{\tau_n(x)}_0
  = \vertiii{\sum_{j=0}^{\pos(\Delta_{n+1},0)-1} x_j \tau_n(e_j)}_0
  \leq \sum_{j=0}^{\pos(\Delta_{n+1},0)-1} |x_j| \cdot \vertiii{e_j}_{N_n+1}
  = \vertiii{x}_{N_n+1}.\qedhere
\]
\end{proof}

\section{Heads}
Recall that
\[
  H_n
  = \Span\left\lbrace e_j:
    j \leq \pos(\Delta_{n+1},0)-1
  \right\rbrace.
\]
We define now the projection (truncation)
\[
  \pi_n\colon s^\N\to H_n
\]
by the formula
\begin{equation}
\label{def:pi_n}
  \pi_n\left(\sum_{j=0}^\infty x_je_{j} \right)
  = \sum_{j=0}^{\pos(\Delta_{n+1}, 0)-1} x_je_{j}.
\end{equation}
Recall that in \eqref{eq:compacts} we defined compact sets
\begin{equation*}
  K_n
  = \left\lbrace
    y \in H_n:
      \vertiii{y}_0 \leq 1
      \text{ and }
      \vertiii{\tau_n(y)}_0\geq \frac{1}{2}
  \right\rbrace.
\end{equation*}
\begin{proposition}
\label{prop:head}
Let $N\in \N$ and
\[
  0 \neq x = \sum_{i,j=0}^\infty x_{i,j}e_{i,j}\in s^\N
\]
be such that
\begin{equation}
\label{eq:head}
  \left|(x_{i,k_0})_{\vphantom{j}i=0}^\infty\right|_0 = 2
\end{equation}
for some $k_0\in \N$.
Let $\left(n_l\right)_{l=0}^\infty$ be an increasing sequence of natural numbers
such that $N_{n_l}=N$ for $l\in \N$. Then for $l$ large enough we have that
\[
  \pi_{n_l}(x)\in\bigcup_{m=1}^\infty mK_{n_l}.
\]
\end{proposition}
\begin{proof}
We can write
\[
  x = \sum_{j=0}^\infty x'_j e_j
\]
for some sequence $(x'_j)_{j=0}^\infty$ of scalars.
Let
\begin{equation}
  \label{eq:im1}
  C_{n_l}
  = \vertiii{\pi_{n_l}(x)}_0
  = \vertiii{\sum_{j=0}^{\pos(\Delta_{n_l+1}, 0)-1} x'_j e_j}_0.
\end{equation}
From \eqref{eq:head} it easily follows that
\begin{equation}
  \label{eq:im2}
  C_{n_l}
  \geq 1 \text{ for $l$ large enough}.
\end{equation}
Observe that
\begin{align*}
  \vertiii{\tau_{n_l}(\pi_{n_l}(x))}_0
  & = \vertiii{\tau_{n_l}\left(\sum_{j=0}^{\pos(\Delta_{n_l+1}, 0)-1} x'_je_{j}\right)}_0
  = \vertiii{\tau_{n_l}\left(\sum_{j=0}^{\pos(a_{n_l},0)-1} x'_je_{j}\right)
    + \tau_{n_l}\left(\sum_{j=\pos(a_{n_l},0)}^{\pos(\Delta_{n_l+1}, 0)-1} x'_je_{j}\right)}_0\\
  & \overset{\mathclap{\ref{lemma:values_tau}}}{=}
    \vertiii{\sum_{j=0}^{\pos(a_{n_l},0)-1} x'_je_{j}
    + \tau_{n_l}\left(\sum_{j=\pos(a_{n_l},0)}^{\pos(\Delta_{n_l+1}, 0)-1} x'_je_{j}\right)}_0\\
  & \overset{\mathclap{\ref{prop:cont_tau}}}{\geq}
    \vertiii{\sum_{j=0}^{\pos(a_{n_l},0)-1} x'_je_{j}}_0
    - \vertiii{\sum_{j=\pos(a_{n_l},0)}^{\pos(\Delta_{n_l+1}, 0)-1} x'_je_{j}}_{N+1}\\
  & \geq
    \vertiii{\sum_{j=0}^{\pos(\Delta_{n_l+1}, 0)-1} x'_je_{j}}_0
    - \vertiii{\sum_{j=\pos(a_{n_l},0)}^{\pos(\Delta_{n_l+1}, 0)-1} x'_je_{j}}_{0}
    - \vertiii{\sum_{j=\pos(a_{n_l},0)}^{\pos(\Delta_{n_l+1}, 0)-1} x'_je_{j}}_{N+1}\\
  & \geq
    C_{n_l} - 2 \vertiii{\sum_{j=\pos(a_{n_l},0)}^{\pos(\Delta_{n_l+1}, 0)-1} x'_je_{j}}_{N+1}.
\end{align*}
It is clear that
\[
  \lim_{l\to\infty}
  2 \vertiii{\sum_{j=\pos(a_{n_l},0)}^{\pos(\Delta_{n_l+1},0)-1} x'_je_{j}}_{N+1} = 0.
\]
Therefore, since $C_{n_l}\geq 1$ for $l$ large enough, from the above inequality we get that for $l$ large enough we have
\begin{equation}
\label{eq:im3}
  \vertiii{\tau_{n_l}(\pi_{n_l}(x))}_0 \geq \frac{C_{n_l}}{2}.
\end{equation}
Thus from \eqref{eq:im1}, \eqref{eq:im2} and \eqref{eq:im3}
it follows that for $l$ large enough
\[
  \pi_{n_l}(x)\in\bigcup_{m=1}^\infty mK_{n_l}.\qedhere
\]
\end{proof}

From the above proposition and Corollary \ref{cor:heads} we get:
\begin{corollary}
  \label{cor:last}
  Let $N\in \N$ and
  \[
    0 \neq x = \sum_{i,j=0}^\infty x_{i,j}e_{i,j}\in s^\N
  \]
  be such that
  \begin{equation*}
    \label{eq:head1}
    \left|(x_{i,k_0})_{\vphantom{j}i=0}^\infty\right|_0 = 2
  \end{equation*}
  for some $k_0\in \N$.
  Let $(n_l)_{l=0}^\infty$ be an increasing sequence of natural numbers
  such that $N_{n_l} = N$ for $l \in \N$. Then for every $l$ large enough
  there exists a polynomial
  \[
    P(t)=\sum_{i=1}^{\pos(\Delta_{n_l+1},0)}c_it^i
    \text{ with }
    \sum_{i=1}^{\pos(\Delta_{n_l+1},0)}|c_i|\leq D_{n_l}
  \]
  and
  \[
    \vertiii{P(T)\pi_{n_l}(x)-e_0}_{N}\leq 3.
  \]
\end{corollary}

\section{Tails}
Observe that (heuristically) if $j$ is much larger then $i$, then $e_{j+i}$ is always in the same column as $e_j$
or in one of the adjacent ones. More precisely:
\begin{lemma}
\label{lemma:snail}
If $j\geq \pos(\Delta_{n+1},0)$ and $1\leq i\leq \pos(\Delta_{n+1},0)$,
then for every $N\in \N$ we have
\[
  \|e_{i+j}\|_N \leq 2^i \|e_j\|_{N+1}.
\]
\end{lemma}
\begin{proof}
Observe that
if $i,j$ are as above and $e_j = e_{p,q}$ for some $p$ and $q$, then
$e_{i+j} = e_{k,l}$, where $k \in \N$ and $l \leq q+1$ (see the definition of
the function $\Next$ and the condition \eqref{cond:pos_bn}).
The desired inequality follows now from Lemma \ref{lemma:estimate}.
\end{proof}
For every $n \in \N$ we define the linear space
\begin{equation}
  \label{def:tails}
  T_n = s^\N \ominus H_n.
\end{equation}
The elements of $T_n$ can be regarded as tails of vectors from $s^\N$.

\begin{proposition}
  \label{prop:tails}
  If $x\in T_n$ and $1 \leq i \leq \pos(\Delta_{n+1},0)$, then
  \[
    \|T^ix\|_{N_n}
    \leq \frac{1}{D_n}\|x\|_{N_n+2}.
  \]
\end{proposition}
\begin{proof}
First we will show that for every $j\geq \pos(\Delta_{n+1},0)$ and
$1 \leq i \leq \pos(\Delta_{n+1},0)$ we have
\begin{equation}
  \label{eq:tails1}
  \|T^ie_j\|_{N_n}
  \leq \frac{1}{D_n}\|e_j\|_{N_n+2}.
\end{equation}
We consider all possible cases for $j$.
\begin{enumerate}[leftmargin=2\parindent]
\item If $\pos(\Delta_{n+1},0)\leq j <\pos(s_{n+1},0)-\pos(\Delta_{n+1},0)$, then
      by \eqref{eq:operator}
      \[
        e_j = \frac{1}{\alpha_j} T^j e_0.
      \]
      Therefore, once again by \eqref{eq:operator},
      \[
        T^i e_j 
        = \frac{1}{\alpha_j} T^{i+j} e_0 
        = \frac{\alpha_{i+j}}{\alpha_j} e_{i+j}.
      \]
      From \eqref{eq:alphas} we get
      \[
        \frac{\alpha_{i+j}}{\alpha_j} = \frac{1}{2^iD_n^i}.
      \]
      Thus, using Lemma \ref{lemma:snail}, we obtain
      \[
        \|T^ie_j\|_{N_n}
        = \frac{1}{2^iD_n^i} \|e_{i+j}\|_{N_n}
        \leq \frac{1}{D_n^i} \|e_j\|_{N_n+1}
        \leq \frac{1}{D_n}   \|e_j\|_{N_n+2}.
      \]
\item If $\pos(s_p,0)-\pos(\Delta_{n+1},0) \leq j < \pos(s_{p+1},0)-\pos(\Delta_{n+1},0)$
      for some $p>n$ and $e_j = e_{k,l}$, where $l \geq N_n+2$,
      then from the construction of the function $\Next$,
      from the definition of $T$ and from \eqref{cond:pos_bn}
      it follows that
      \[
        \|T^ie_j\|_{N_n} = 0
      \]
      and therefore \eqref{eq:tails1} holds.
\item If $\pos(s_p,0)-\pos(\Delta_{n+1},0) \leq j < \pos(s_{p+1},0)-\pos(\Delta_{n+1},0)$
      for some $p > n$ and $e_j = e_{k,l}$, where
      $l \leq N_n+1$, then from the construction of the function $\Next$ and from
      \eqref{eq:N_n} it follows that $k \geq b_{n+1}$.
      We have
      \begin{equation}
        \label{eq:tails2}
        \|e_j\|_{N_n+2} 
        = \|e_{k,l}\|_{N_n+2}
        = A_{N_n+2,k}.
      \end{equation}
      Using Proposition \ref{prop:continuity} we get that
      \begin{align*}
        \|T^ie_j\|_{N_n}
        & \leq \vertiii{T^ie_j}_{N_n}
        \leq 4^i \vertiii{e_j}_{N_n}
        \leq 4^{\pos(\Delta_{n+1},0)} \vertiii{e_j}_{N_n} \\
        & \leq 4^{\pos(\Delta_{n+1},0)}\|e_j\|_{N_n+1}
        = 4^{\pos(\Delta_{n+1},0)} A_{N_n+1,k}\\
        & \overset{\mathclap{\eqref{cond:2bn}}}{\leq}
          \frac{A_{N_n+2,k}}{D_n}
        \overset{\mathclap{\eqref{eq:tails2}}}{=}
          \frac{1}{D_n}\|e_j\|_{N_n+2}.
      \end{align*}
      \end{enumerate}
To finish the proof take $x\in T_n$. Then
\[
  x = \sum_{j=\pos(\Delta_{n+1},0)}^\infty x_je_j
\]
and by the inequality we have just proven we have that
\begin{equation*}
  \|T^ix \|_{N_n}
  \leq \frac{1}{D_n} \sum_{j=\pos(\Delta_{n+1},0)}^\infty |x_j|\|e_j\|_{N_n+2}
  =    \frac{1}{D_n} \|x\|_{N_n+2}.\qedhere
\end{equation*}
\end{proof}

\section{The operator \texorpdfstring{$T$}{T} has no non-trivial invariant subspaces}
After all the hard work done in the previous sections we are ready to prove the main result of the paper.
\begin{theorem*}
  There exists a continuous operator $T \colon s^\N \to s^\N$ which has no non-trivial invariant subspaces.
\end{theorem*}
\begin{proof}
Let $T \colon s^\N \to s^\N$ be the operator defined by \eqref{eq:operator}.
From Proposition \ref{prop:values} it follows that $T$ is a perturbed
weighted forward shift and therefore $e_0$ is a cyclic vector for $T$.
We need to show that every non-zero $x \in s^\N$ is also a cyclic vector for $T$.

Let
\[
  0 \neq x = \left(x_{i,j}\right)_{i,j=0}^\infty \in s^\N
\]
be arbitrary and $k_0$ be the smallest integer for which
\[
  \left|(x_{i,k_0})_{\vphantom{j}i=0}^\infty\right|_0 \neq 0.
\]
Since any non-zero multiple of a cyclic vector is also a cyclic vector, we may assume with no loss of generality that
\begin{equation*}
  \left|(x_{i,k_0})_{\vphantom{j}i=0}^\infty\right|_0 = 2.
\end{equation*}
By Proposition \ref{prop:strategy}, in order to show that $x$ is a cyclic vector for $T$ we need to prove that for every $N \in \N$
there is a polynomial $P$ such that
\[
  \|P(T)x-e_0\|_N\leq 4.
\]
Fix $N \in \N$ and let $(n_l)_{l=0}^\infty$
be an increasing sequence of natural numbers such that $N_{n_l}=N$ for every
$l\in \N$. For every $n \in \N$ let $\pi_n$ be the projection defined by \eqref{def:pi_n}.
It is clear (see \eqref{def:tails}) that
\[
  x-\pi_n(x) \in T_n
\]
and that for every $n$ large enough we have
\begin{equation*}
\|x-\pi_n(x)\|_{N+2} \leq 1.
\end{equation*}
Using this and Corollary \ref{cor:last} we can find $l$ large enough
and a polynomial $\displaystyle P(t) = \sum_{i=1}^{\pos(\Delta_{n_l+1},0)}c_it^i$
such that
\begin{equation}
  \label{eq:end}
 \sum_{i=1}^{\pos(\Delta_{n_l+1},0)}|c_i| \leq D_{n_l},
 \quad
 \vertiii{P(T)\pi_{n_l}(x)-e_0}_{N} \leq 3
 \quad
 \text{and}
 \quad
 \|x-\pi_{n_l}(x)\|_{N+2} \leq 1.
\end{equation}
We have
\begin{align*}
  \|P(T)x-e_0\|_N
  & = \|P(T)\pi_{n_l}(x) + P(T)(x-\pi_{n_l}(x))-e_0\|_N\\
  & \leq \|P(T)\pi_{n_l}(x)-e_0\|_N
          +\|P(T)(x-\pi_{n_l}(x))\|_N\\
  & \leq \vertiii{P(T)\pi_{n_l}(x)-e_0}_{N}+D_{n_l}
         \max_{1 \leq i \leq \pos(\Delta_{n_l+1},0)} \|T^i(x-\pi_{n_l}(x))\|_N \\
  & \overset{\mathclap{\ref{prop:tails}}}{\leq}
      \vertiii{P(T)\pi_{n_l}(x)-e_0}_{N}
       + D_{n_l}\frac{\|x-\pi_{n_l}(x)\|_{N+2}}{D_{n_l}}\\
  & \overset{\mathclap{\eqref{eq:end}}}{\leq}
    3+1=4.\qedhere
\end{align*}
\end{proof}

\bibliographystyle{plain}
\bibliography{baza}

\begin{thebibliography}{10}

\bibitem{MR701260}
A.~Atzmon.
\newblock An operator without invariant subspaces on a nuclear {F}r\'{e}chet
  space.
\newblock {\em Ann. of Math. (2)}, 117(3):669--694, 1983.

\bibitem{MR2533318}
F.~Bayart and E.~Matheron.
\newblock {\em Dynamics of linear operators}, volume 179 of {\em Cambridge
  Tracts in Mathematics}.
\newblock Cambridge University Press, Cambridge, 2009.

\bibitem{MR0473871}
P.~Enflo.
\newblock On the invariant subspace problem in {B}anach spaces.
\newblock In {\em S\'{e}minaire {M}aurey--{S}chwartz (1975--1976) {E}spaces
  {$L^{p}$}, applications radonifiantes et g\'{e}om\'{e}trie des espaces de
  {B}anach, {E}xp. {N}os. 14-15}, page~7. Centre Math., \'{E}cole Polytech.,
  Palaiseau, 1976.

\bibitem{MR892591}
P.~Enflo.
\newblock On the invariant subspace problem for {B}anach spaces.
\newblock {\em Acta Math.}, 158(3-4):213--313, 1987.

\bibitem{MR2863862}
M.~Goli\'{n}ski.
\newblock Invariant subspace problem for classical spaces of functions.
\newblock {\em J. Funct. Anal.}, 262(3):1251--1273, 2012.

\bibitem{MR3077885}
M.~Goli\'{n}ski.
\newblock Operator on the space of rapidly decreasing functions with all
  non-zero vectors hypercyclic.
\newblock {\em Adv. Math.}, 244:663--677, 2013.

\bibitem{MR1483073}
R.~Meise and D.~Vogt.
\newblock {\em Introduction to functional analysis}, volume~2 of {\em Oxford
  Graduate Texts in Mathematics}.
\newblock The Clarendon Press, Oxford University Press, New York, 1997.
\newblock Translated from the German by M. S. Ramanujan and revised by the
  authors.

\bibitem{MR3866905}
Q.~Menet.
\newblock Invariant subspaces for non-normable {F}r\'{e}chet spaces.
\newblock {\em Adv. Math.}, 339:495--539, 2018.

\bibitem{MR749447}
C.~J. Read.
\newblock A solution to the invariant subspace problem.
\newblock {\em Bull. London Math. Soc.}, 16(4):337--401, 1984.

\bibitem{MR806634}
C.~J. Read.
\newblock A solution to the invariant subspace problem on the space {$l_1$}.
\newblock {\em Bull. London Math. Soc.}, 17(4):305--317, 1985.

\bibitem{MR959046}
C.~J. Read.
\newblock The invariant subspace problem for a class of {B}anach spaces. {II}.
  {H}ypercyclic operators.
\newblock {\em Israel J. Math.}, 63(1):1--40, 1988.

\bibitem{MR950973}
C.~J. Read.
\newblock The invariant subspace problem on a class of nonreflexive {B}anach
  spaces. {I}.
\newblock In {\em Geometric aspects of functional analysis (1986/87)}, volume
  1317 of {\em Lecture Notes in Math.}, pages 1--20. Springer, Berlin, 1988.

\end{thebibliography}
\end{document}